\documentclass[12pt]{amsart}
\usepackage{hyperref}

\newtheorem{theorem}{Theorem}[section]
\newtheorem{definition}[theorem]{Definition}

\newtheorem{proposition}[theorem]{Proposition}
\newtheorem{corollary}[theorem]{Corollary}

\theoremstyle{remark}
\newtheorem{remark}[theorem]{Remark}

\newcommand\NN{{\mathbb N}}
\newcommand\TT{{\mathbb T}}
\newcommand\ZZ{{\mathbb Z}}
\newcommand\Tr{{\rm Tr}}
\newcommand\ee{{\epsilon}}
\newcommand\bee{\boldmath{\epsilon}}
\newcommand\ii{\mathbf{i}}
\newcommand\Cp{C}
\newcommand\ff{\varphi}

\begin{document}

\title{Stochastic aspects of easy quantum groups}

\author{Teodor Banica}
\address{T.B.: Department of Mathematics, Cergy-Pontoise
University, 2 avenue Chauvin, 95302 Cergy-Pontoise, France. {\tt teodor.banica@u-cergy.fr}}

\author{Stephen Curran}
\address{S.C.: Department of Mathematics, University of California,
Berkeley, CA 94720, USA. {\tt curransr@math.berkeley.edu}}

\author{Roland Speicher}
\address{R.S.: Department of Mathematics and Statistics,
Queen's University, Jeffery Hall, Kingston, Ontario K7L 3N6,
Canada. {\tt speicher@mast.queensu.ca}}

\subjclass[2000]{60B15 (16T30, 46L54)} \keywords{random matrices, quantum group,
noncrossing partition}

\begin{abstract}
We consider several orthogonal quantum groups satisfying the ``easiness'' assumption
axiomatized in our previous paper. For each of them we discuss the computation of the
asymptotic law of ${\rm Tr}(u^k)$ with respect to the Haar measure, $u$ being the
fundamental representation. For the classical groups $O_n,S_n$ we recover in this way
some well-known results of Diaconis and Shahshahani.
\end{abstract}

\maketitle

\section*{Introduction}

The present paper is a continuation of our previous work \cite{bsp} on easy quantum
groups. We will present a concrete application of our formalism: a unified approach plus
quantum group extension of some results of Diaconis-Shahshahani \cite{dsh}.

The objects of interest will be the compact quantum groups satisfying $S_n\subset
G\subset O_n^+$. Here $O_n^+$ is the free analogue of the orthogonal group, constructed
by Wang in \cite{wa1}, and for the compact quantum groups we use Woronowicz's formalism
in \cite{wo1}.

As in \cite{bsp}, we restrict attention to the ``easy'' case. The easiness assumption,
essential to our considerations, roughly states that the tensor category of $G$ should be
spanned by certain partitions, coming from the tensor category of $S_n$. This might look
like a quite technical condition, but in our opinion this provides a good framework for
understanding certain probabilistic and representation theory aspects of orthogonal
quantum groups.

There are 14 natural examples of easy quantum groups found in \cite{bsp}. The list is as
follows:
\begin{enumerate}
\item Groups: $O_n,S_n,H_n,B_n,S_n',B_n'$.

\item Free versions: $O_n^+,S_n^+,H_n^+,B_n^+,S_n'^+,B_n'^+$.

\item Half-liberations: $O_n^*,H_n^*$.
\end{enumerate}

The 4 ``primed'' versions above are rather trivial modifications of their ``unprimed''
versions, corresponding to taking a product with a copy of $\mathbb{Z}_2$. We will focus
then on the remaining 10 examples in this paper. In addition to the 14 examples listed
above, there are two infinite series $H_n^{(s)}$ and $H_n^{[s]}$, described in
\cite{bcs1}, which are related to the complex reflection groups $H_n^s=\mathbb{Z}_s\wr
S_n$.

Our motivating belief, already present in \cite{bsp}, is that ``any result which holds
for $O_n,S_n$ should have a suitable extension to all the easy quantum groups''. This is
of course a quite vague statement, its precise target being actually formed by a number
of questions at the borderline between representation theory and probability theory.

It was suggested in \cite{bsp} that a first such application might come from the results
of Diaconis of Shahshahani in \cite{dsh}, regarding the groups $O_n,S_n$. We will show in
this paper that this is indeed the case:
\begin{enumerate}
\item The problematics makes indeed sense for all easy quantum groups.

\item There is a global approach to it, by using partitions and cumulants.

\item The new computations lead to a number of interesting conclusions.
\end{enumerate}

As a first example, consider the orthogonal group $O_n$, with fundamental representation
denoted $u$. The results in \cite{dsh}, that we will recover as well by using our
formalism, state that the asymptotic variables $u_k=\lim_{n\to\infty} {\rm Tr}(u^k)$ are
real Gaussian and independent, with variance $k$ and mean $0$ or 1, depending on whether
$k$ is odd or even.

In the case of $O_n^+$, however, the situation is quite different: the variables $u_k$
are free, as one could expect, but they are semicircular at $k=1,2$, and circular at
$k\geq 3$.

Summarizing, in the orthogonal case we have the following table:
\begin{center}
\begin{tabular}[t]{|l|l|l|l|l|}
\hline Variable&$O_n$&$O_n^+$\\
\hline
\hline $u_1$&real Gaussian&semicircular\\
\hline $u_2$&real Gaussian&semicircular\\
\hline $u_k$ ($k\geq 3$)&real Gaussian&circular\\
\hline
\end{tabular}
\end{center}
\medskip

In the symmetric case the situation is even more surprising, with the Poisson variables
from the classical case replaced by several types of variables:
\begin{center}
\begin{tabular}[t]{|l|l|l|l|l|}
\hline Variable&$S_n$&$S_n^+$\\
\hline
\hline $u_1$&Poisson&free Poisson\\
\hline $u_2-u_1$&Poisson&semicircular\\
\hline $u_k-u_1$ ($k\geq 3$)&sum of Poissons&circular\\
\hline
\end{tabular}
\end{center}
\medskip

We will present as well similar computations for the groups $H_n,B_n$, for their free
analogues $H_n^+,B_n^+$, for the half-liberated quantum groups $O_n^*$, $H_n^*$, as well
as for the series $H_n^{(s)}$. The calculations in the latter case rely essentially on
Diaconis-Shahshahani type results for the complex reflection groups
$H_n^s=\mathbb{Z}_s\wr S_n$.

The challenging question, that will eventually be left open, is to find a formal
``eigenvalue'' interpretation for all the quantum group results.

The paper is organized as follows. After a short Section 0 with notational remarks, we
recall the basic definitions and facts about easy quantum groups in Section 1. In Section
2, we recall the Weingarten formula for our easy quantum groups and use it to derive a
formula for the moments of traces of powers. In Section 3, this is refined to a formula
for corresponding cumulants, in the classical and the free cases. In Sections 4-7, this
will then be used to study the orthogonal, bistochastic, symmetric, and hyperoctahedral
classical and quantum groups, respectively. Section 8 deals with the half-liberated
quantum groups $O_n^*$ and $H_n^*$. The results in these cases will rely on the
observation that these half-liberated quantum groups are in some sense orthogonal
versions of classical groups, $U_n$ for $O_n^*$ and $H^{\infty}$ for $H_n^*$. The main
calculations will take place for these classical groups. The same ideas work actually for
the half-liberated series $H_n^{(s)}$, by considering those as orthogonal versions of the
complex reflection groups $H_n$. One of the main results in Section 8 is a
Diaconis-Shahshahani type result for those classical reflection groups. In Section 9, we
finish with some concluding remarks and open problems.

\subsection*{Acknowledgements}

The work of T.B. was supported by the ANR grants ``Galoisint'' and ``Granma'', and the
work of R.S. was supported by a Discovery grant from NSERC.

\section*{0. Notation}

\subsection*{Quantum groups}
As in our previous work \cite{bsp}, the basic object under consideration will be a
compact quantum group $G$. The concrete examples of such quantum groups include the usual
compact groups $G$, and, to some extent, the duals of discrete groups $\widehat{\Gamma}$.
In the general case, however, $G$ is just a fictional object, which exists only via its
associated Hopf $C^*$-algebra of ``complex continuous functions'', generically denoted
$A$.

The fact that $G$ itself doesn't exist is not really an issue, because many advanced
tools coming from algebra, analysis and geometry are available. In fact, to the
well-known criticism stating that ``quantum groups don't exist'', our answer would be
that ``classical groups exist, indeed, but is their existence property the most
important?''.

For simplicity of notation, we will rather use the quantum group $G$ instead of the Hopf
algebra $A$. For instance we will write integrals of the following type:
$$\int_Gu_{i_1j_1}\ldots u_{i_kj_k}\,du$$

The value of this integral is of course the complex number obtained by applying the Haar
functional $\varphi:A\to\mathbb C$ to the well-defined quantity $u_{i_1j_1}\ldots
u_{i_kj_k}\in A$.

We will use the quantum group notation depending on the setting: in case where this can
lead to confusion, we will rather switch back to the Hopf algebra notation.

\subsection*{Partitions}
For the notations and basic facts around the set of all and non-crossing partitions we
refer to \cite{nsp} and our previous papers \cite{bsp,bcs1,bcs2}. We will in particular
use the following notations.

$P_k$ denotes the set of partitions of the set $\{1,\dots,k\}$. $1_k$ denotes the maximal
element in $P_k$, which consists only of one block. For a partition $\pi\in P_k$ we
denote by $\vert \pi\vert$ the number of blocks of $\pi$.

With $\ii$ we will usually denote multi-indices $\ii=(i_1,\dots,i_k)$. Often, the
constraints in sums for such indices are given in terms of their kernel, $\ker\ii=\ker
(i_1,\dots,i_k)$. This is the partition in $P_k$ determined as follows:
$$\text{$s$ and $t$ are in the same block of $\ker\ii$}\qquad
\Longleftrightarrow \qquad i_s=i_t.$$

For given $k_1,\dots,k_r\in \NN$, $k:=\sum k_i$, we will denote by $\gamma\in S_k$ the
permutation with cycles $(1,\ldots,k_1)$, $(k_1+1,\ldots,k_1+k_2)$, \ldots
,$(k-k_s+1,\ldots,k)$. If we have, in addition, a partition $\sigma\in P_r$, then
$\sigma^\delta$ will denote the canonical lift of $\sigma$ from $P_r$ to $P_k$,
associated to $\gamma$. Thus, $\sigma^\gamma$ is that partition which we get from
$\sigma$ by replacing each $j\in\{1,\dots,r\}$ by the $j$-th cycle of $\gamma$, i.e.,
$\sigma^\gamma\geq \gamma$ and the $i$-th and the $j$-th cycle of $\gamma$ are in the
same block of $\sigma^\gamma$ if and only if $i$ and $j$ are in the same block of
$\sigma$.

As an example, let $\gamma=(1) (2,3,4) (5,6)$. Consider now $\sigma=\{(1,2),(3)\}\in
P_3$:

\setlength{\unitlength}{0.5cm}
\[
  \begin{picture}(3,1.5)\thicklines
  \put(0,0){\line(0,1){1}}
  \put(0,0){\line(1,0){1}}
  \put(1,0){\line(0,1){1}}
  \put(2,0){\line(0,1){1}}
  \put(0,1.5){\makebox(0,0){$1$}}
  \put(1,1.5){\makebox(0,0){$2$}}
  \put(2.0,1.5){\makebox(0,0){$3$}}
  \end{picture}
\]
Then $\sigma^\gamma$ is given by making the replacements $1 \to 1$, $2 \to 2,3,4$ and $3
\to 5,6$,
 \setlength{\unitlength}{0.5cm}
\[
  \begin{picture}(5,2)\thicklines
  \put(0,0){\line(0,1){1}}
  \put(0,0){\line(1,0){3}}
\put(4,0){\line(1,0){1}}
  \put(1,0){\line(0,1){1}}
  \put(2,0){\line(0,1){1}}
  \put(3,0){\line(0,1){1}}
  \put(4,0){\line(0,1){1}}
  \put(5,0){\line(0,1){1}}
  \put(0,1.5){\makebox(0,0){$1$}}
  \put(1,1.5){\makebox(0,0){$2$}}
  \put(2,1.5){\makebox(0,0){$3$}}
  \put(3,1.5){\makebox(0,0){$4$}}
  \put(4,1.5){\makebox(0,0){$5$}}
  \put(5,1.5){\makebox(0,0){$6$}}
  \end{picture}
\]
thus $\sigma^\gamma=\{(1,2,3,4),(5,6)\}\in P_6$.

Note that in \cite{nsp} the notation $\hat \sigma$ was used for $\sigma^\gamma$.

\section{Easy quantum groups}

In this section and in the next one we briefly recall some notions and results from
\cite{bsp,bcs1}.

Consider first a compact group satisfying $S_n\subset G\subset O_n$. That is, $G\subset
O_n$ is a closed subgroup, containing the subgroup $S_n\subset O_n$ formed by the
permutation matrices.

Let $u,v$ be the fundamental representations of $G,S_n$. By functoriality we have
$Hom(u^{\otimes k},u^{\otimes l})\subset Hom(v^{\otimes k},v^{\otimes l})$, for any
$k,l$. On the other hand, the Hom-spaces for $v$ are well-known: they are spanned by
certain explicit operators $T_p$, with $p$ belonging to $P(k,l)$, the set of partitions
between $k$ points and $l$ points. More precisely, if $e_1,\ldots,e_n$ denotes the
standard basis of $\mathbb C^n$, the formula of $T_p$ is as follows:
$$T_p(e_{i_1}\otimes\ldots\otimes e_{i_k})=\sum_{j_1,\ldots,j_l}
\delta_p\begin{pmatrix}i_1&\ldots&i_k\\ j_1&\ldots&j_l
\end{pmatrix}e_{j_1}\otimes\ldots\otimes e_{j_l}$$

Here the $\delta$ symbol on the right is 0 or 1, depending on whether the indices ``fit''
or not, i.e. $\delta=1$ if all blocks of $p$ contains equal indices, and $\delta=0$ if
not.

We conclude from the above discussion that the space $Hom(u^{\otimes k},u^{\otimes l})$
consists of certain linear combinations of operators of type $T_p$, with $p\in P(k,l)$.

We call $G$ ``easy'' if its tensor category is spanned by partitions.

\begin{definition}\label{def:easy}
A compact group $S_n\subset G\subset O_n$ is called \emph{easy} if there exist sets
$D(k,l)\subset P(k,l)$ such that $Hom(u^{\otimes k},u^{\otimes l})=span(T_p|p\in
D(k,l))$, for any $k,l$.
\end{definition}

It follows from the axioms of tensor categories that the collection of sets $D(k,l)$ must
be closed under certain categorical operations, namely the vertical and horizontal
concatenation, and the upside-down  turning. The corresponding algebraic structure formed
by the sets $D(k,l)$, axiomatized in \cite{bsp}, is called ``category of partitions''.

We denote by $H_n=\mathbb Z_2\wr S_n$ the hyperoctahedral group, formed by the monomial
(i.e. permutation-like) matrices having $\pm 1$ nonzero entries. The bistochastic group,
$B_n\simeq O_{n-1}$, is by definition formed by the matrices in $O_n$ having sum 1 on
each row and each column. Finally, the modified symmetric and bistochastic groups are by
definition $S_n'=\mathbb Z_2\times S_n$ and $B_n'=\mathbb Z_2\times B_n$, both viewed as
subgroups of $O_n$. See \cite{bsp}.

\begin{theorem}\label{thm:easy-groups}
There are exactly $6$ easy orthogonal groups, namely:
\begin{enumerate}
\item $O_n$: the orthogonal group.

\item $S_n$: the symmetric group.

\item $H_n$: the hyperoctahedral group.

\item $B_n$: the bistochastic group.

\item $S_n'$: the modified symmetric group.

\item $B_n'$: the modified bistochastic group.
\end{enumerate}
\end{theorem}

\begin{proof}
As explained in \cite{bsp}, this follows from a 6-fold classification result for the
corresponding categories of partitions, which are as follows:

(1) $P_o$: all pairings.

(2) $P_s$: all partitions.

(3) $P_h$: partitions with blocks of even size.

(4) $P_b$: singletons and pairings.

(5) $P_{s'}$: all partitions (even part).

(6) $P_{b'}$: singletons and pairings (even part).
\end{proof}

Let us discuss now the free analogue of the above results. Let $O_n^+,S_n^+$ be the free
orthogonal and symmetric quantum groups, corresponding to the Hopf algebras
$A_o(n),A_s(n)$ constructed by Wang in \cite{wa1}, \cite{wa2}. Here, and in what follows,
we use Woronowicz's compact quantum group formalism in \cite{wo1}, cf. section 0 above.

We have $S_n\subset S_n^+$, so by functoriality the Hom-spaces for $S_n^+$ appear as
subspaces of the corresponding Hom-spaces for $S_n$. The Hom-spaces for $S_n^+$ have in
fact a very simple description: they are spanned by the operators $T_p$, with $p\in
NC(k,l)$, the set of noncrossing partitions between $k$ upper points and $l$ lower
points.

We have the following ``free analogue'' of Definition \ref{def:easy}.

\begin{definition}\label{def:freeqg}
A compact quantum group $S_n^+\subset G\subset O_n^+$ is called \emph{free} if there
exist sets $D(k,l)\subset NC(k,l)$ such that $Hom(u^{\otimes k},u^{\otimes
l})=span(T_p|p\in D(k,l))$, for any $k,l$.
\end{definition}

In this definition, the word ``free'' has of course a quite subtle meaning, to be fully
justified later on. For the moment, let us just record the fact that the passage from
Definition \ref{def:easy} to Definition \ref{def:freeqg} is basically done by
``restricting attention to the noncrossing partitions'', which, according to \cite{sp1},
should indeed lead to freeness.

As in the classical case, the sets of partitions $D(k,l)$ must be stable under certain
categorical operations, coming this time from the axioms in \cite{wo2}. The corresponding
algebraic structure, axiomatized in \cite{bsp}, is called ``category of noncrossing
partitions''.

We denote by $H_n^+$ the hyperoctahedral quantum group, constructed in \cite{bbc}, and by
$B_n^+,S_n'^+,B_n'^+$ the free analogues of the groups $B_n,S_n',B_n'$, constructed in
\cite{bsp}.

\begin{theorem}\label{thm:free-quantum-groups}
There are exactly $6$ free orthogonal quantum groups, namely:
\begin{enumerate}
\item $O_n^+$: the orthogonal quantum group.

\item $S_n^+$: the symmetric quantum group.

\item $H_n^+$: the hyperoctahedral quantum group.

\item $B_n^+$: the bistochastic quantum group.

\item $S_n'^+$: the modified symmetric quantum group.

\item $B_n'^+$: the modified bistochastic quantum group.
\end{enumerate}
\end{theorem}

\begin{proof}
As explained in \cite{bsp}, this follows from a 6-fold classification result for the
corresponding categories of noncrossing partitions, which are as follows:

(1) $NC_o$: all noncrossing pairings.

(2) $NC_s$: all noncrossing partitions.

(3) $NC_h$: noncrossing partitions with blocks of even size.

(4) $NC_b$: singletons and noncrossing pairings.

(5) $NC_{s'}$: all noncrossing partitions (even part).

(6) $NC_{b'}$: singletons and noncrossing pairings (even part).
\end{proof}

Observe the symmetry between Theorem \ref{thm:easy-groups} and Theorem
\ref{thm:free-quantum-groups}: this corresponds to the ``liberation'' operation for
orthogonal Lie groups, further investigated in \cite{bsp}.

Consider now the general situation where we have a compact quantum group satisfying
$S_n\subset G\subset O_n^+$. Once again, we can ask for the tensor category of $G$ to be
spanned by certain partitions, coming from the tensor category of $S_n$.

\begin{definition}
A compact quantum group $S_n\subset G\subset O_n^+$ is called \emph{easy} if there exist
sets $D(k,l)\subset P(k,l)$ such that $Hom(u^{\otimes k},u^{\otimes l})=span(T_p|p\in
D(k,l))$, for any $k,l$.
\end{definition}

As a first remark, this definition generalizes at the same time Definition \ref{def:easy}
and Definition \ref{def:freeqg}. In fact, the easy quantum groups $S_n\subset G\subset
O_n^+$ satisfying the extra assumption $G\subset O_n$ are precisely the easy groups, and
those satisfying the extra assumption $S_n^+\subset G$ are precisely the free quantum
groups.

Once again, the sets of partitions $D(k,l)$ must be stable under certain categorical
operations, coming from the axioms in \cite{wo2}. The corresponding algebraic structure,
axiomatized in \cite{bsp}, is called ``full category of partitions''.

We already know that the easy orthogonal quantum groups include the 6 easy groups in
Theorem \ref{thm:easy-groups}, and the 6 free quantum groups in Theorem
\ref{thm:free-quantum-groups}. In \cite{bsp}, two more canonical examples of easy quantum
groups were found. These extra two examples are the quantum groups $O_n^*,H_n^*$,
obtained as ``half-liberations'' of $O_n,H_n$. The idea is as follows: instead of
removing the commutativity relations of type $ab=ba$ from the standard presentation of
$C(G)$, which would produce the algebra $C(G^+)$, we replace these commutativity
relations by the weaker relations $abc=cba$, which produce by definition the algebra
$C(G^*)$. See \cite{bsp}, where also the following theorem is proved.

\begin{theorem}
The following are easy orthogonal quantum groups:
\begin{enumerate}
\item $O_n^*$: the half-liberated orthogonal group.

\item $H_n^*$: the half-liberated hyperoctahydral group.

\end{enumerate}

These correspond to the following categories of partitions:
\begin{enumerate}
\item $E_o$: pairings with each pair connecting an odd and an even number.
\item
$E_h$: partitions with each block having the same number of odd and even legs.
\end{enumerate}

\end{theorem}

In addition to the 14 natural examples defined above, there are also two infinite
``hyperoctahedral'' series $H_n^{(s)}$ and $H_n^{[s]}$. These are introduced in
\cite{bcs1}, where we give also some partial classification results for easy quantum
groups, with the conjectural conclusion that the easy quantum groups consists of the 14
natural examples, and a multi-parameter ``hyperoctahedral'' series unifying $H_n^{(s)}$
and $H_n^{[s]}$. In the present paper we will mainly consider the natural easy quantum
groups. Since the modified permutation and bistochastic groups $S'_n$, $B'_n$ and their
free versions $S'^+_n$, $B'^+_n$ are somewhat trivial modifications of their ``unprimed''
versions, we will not consider them any further, and thus restrict our attention to the
easy groups $O_n$, $S_n$, $H_n$, $B_n$, the free quantum groups $O_n^+$, $S_n^+$,
$H_n^+$, $B_n^+$, and the half-liberated quantum groups $O_n^*$, $H_n^*$. Since it turns
out that also the series $H_n^{(s)}$ (which includes $H_n$ and $H_n^*$ for $s=2$ and
$s=\infty$, respectively) can be treated by the same methods as $H_n$ and $H_n^*$, we
will also include this series in our considerations. The series $H_n^{[s]}$, on the other
side, is quite elusive at the moment, and it seems that one needs new tools to address
them. We plan to return to this question after completing the full classification of all
easy quantum groups.

Let us finally describe also the quantum groups $H_n^{(s)}$ in terms of their category of
partitions. For more details, as well as the proof of the following theorem, see
\cite{bcs1}.

\begin{theorem}\label{thm:hyper-series}
For $s\in\{2,3,4,\dots,\infty\}$, $H_n^{(s)}$ is an easy quantum group, and its
associated category $E_h^s$ is that of the ``$s$-balanced'' partitions, i.e. partitions
satisfying the following conditions:
\begin{enumerate}
\item The total number of legs is even.

\item In each block, the number of odd legs equals the number of even legs, modulo $s$.
(For $s=\infty$, this means that the number of odd legs equals the number of even legs.)
\end{enumerate}
\end{theorem}

\section{Moments of powers}

In this section we discuss the computation of the asymptotic joint distribution of the
variables ${\rm Tr}(u^k)$, generalizing the fundamental character $\chi={\rm Tr}(u)$.

Let us first recall some general results from \cite{bsp}. Let $G$ be an easy orthogonal
quantum group, and denote by $D_k\subset P(0,k)$ the corresponding sets of diagrams,
having no upper points. We define the Gram matrix to be $G_{kn}(p,q)=n^{\vert p\vee
q\vert}$, where $\vert p\vert$ denotes the number of blocks of the partition $p$. The
Weingarten matrix is by definition its inverse, $W_{kn}=G_{kn}^{-1}$. In order for this
inverse to exist, $n$ has to be big enough, and the assumption $n\geq k$ is sufficient.
See \cite{bsp}.

We use the notation for integrals from section 0 above.

\begin{theorem}\label{thm:Weingarten-formula}
The Haar integration over $G$ is given by
$$\int_Gu_{i_1j_1}\ldots u_{i_kj_k}\,du=\sum_{
\substack{p,q \in D_k\\ p \leq \ker \mathbf i\\ q \leq \ker \mathbf j} }W_{kn}(p,q)$$
\end{theorem}

\begin{proof}
This is proved in \cite{bsp}, the idea being that the integrals on the left,taken
altogether, form the orthogonal projection on $Fix(u^{\otimes k})=span(D_k)$.
\end{proof}

The above formula can be used for computing the asymptotic moments, cumulants and
spectral densities of the truncated characters $\chi_t=\sum_{i=1}^{[tn]}u_{ii}$. Without
getting into details, let us just mention that the laws of truncated characters are given
as follows (see \cite{bsp,bcs1,bcs2} for notions and proofs):

\begin{enumerate}
\item For $O_n,S_n,H_n,B_n$ we get the Gaussian, Poisson, Bessel
and shifted Gaussian laws, which form convolution semigroups.

\item For $O_n^+,S_n^+,H_n^+,B_n^+$ we get the semicircular, free
Poisson, free Bessel and shifted semicircular laws, which form free convolution
semigroups.

\item For $S_n',H_n',S_n'^+,H_n'^+$ we get the symmetrized versions of
the corresponding laws in (1,2), which do not form convolution or free convolution
semigroups, because the canonical copy of $\mathbb Z_2$ gives rise to a correlation.

\item
For $O_n^*,H_n^*$ we get squeezed versions of the complex Gaussian and Bessel measure,
which form ``half-independent '' convolution semigroups.
\end{enumerate}

We turn now to our main problem: the computation of the asymptotic laws of powers ${\rm
Tr}(u^k)$ with $k\in\mathbb N$, generalizing the usual characters $\chi={\rm Tr}(u)$. In
the classical case these laws, computed by Diaconis and Shahshahani in \cite{dsh}, can be
of course understood in terms of the asymptotic behavior of the eigenvalues of the random
matrices $u\in G$.

As in \cite{dsh}, we will be actually interested in the more general problem consisting
in computing the joint asymptotic law of the variables ${\rm Tr}(u^k)$, with $k\in\mathbb
N$ varying. In order to deal with these joint laws, it is convenient to use the following
definition.

\begin{definition}\label{def:trace-permutation}
Associated to $k_1,\ldots,k_s\in\mathbb N$ is the \emph{trace permutation} $\gamma\in
S_k$, with $k=\Sigma k_i$, having as cycles $(1,\ldots,k_1)$, $(k_1+1,\ldots,k_1+k_2)$,
\ldots ,$(k-k_s+1,\ldots,k)$.
\end{definition}

Our first general result concerns the joint moments of the variables ${\rm Tr}(u^k)$, and
is valid for any easy quantum groups.

We denote by $\gamma(q)$ the partition given by $i\sim_qj$ iff
$\gamma(i)\sim_{\gamma(q)}\gamma(j)$.

\begin{theorem}\label{thm:moments}
Let $G$ be an easy quantum group. Consider $s\in\NN$, $k_1,\dots,k_s\in\NN$,
$k:=\sum_{i=1}^s k_i$, and denote by $\gamma\in S_k$ the trace permutation associated to
$k_1,\ldots,k_s$. Then we have, for any $n$ such that $G_{nk}$ is invertible,
\begin{equation}\label{eq:moments}
\int_G {\rm Tr}(u^{k_1})\ldots {\rm Tr}(u^{k_s})\,du=\#\{p\in D_k|p=\gamma(p)\} +O(1/n).
\end{equation}

If $G$ is a classical easy group, then \eqref{eq:moments} is exact, without any lower
order corrections in $n$.
\end{theorem}

\begin{proof}
We denote by $I$ the integral to be computed. According to the definition of $\gamma$, we
have the following formula:
\begin{eqnarray*}
I
&=&\int_G {\rm Tr}(u^{k_1})\ldots {\rm Tr}(u^{k_s})\,du\\
&=&\sum_{i_1\ldots i_k}\int_G(u_{i_1i_2}\ldots u_{i_ki_1})\ldots (u_{i_{k-k_s+1}i_{k-k_s+2}}\ldots u_{i_ki_{k-k_s+1}})\\
&=&\sum_{i_1\ldots i_k}\int_Gu_{i_1i_{\gamma(1)}}\ldots u_{i_ki_{\gamma(k)}}
\end{eqnarray*}

We use now the Weingarten formula from Theorem \ref{thm:Weingarten-formula}. We get:
\begin{eqnarray*}
I &=&\sum_{i_1\ldots i_k=1}^n\sum_{\substack{p,q\in D_k\\
p\leq \ker \ii, q\leq \ker \ii\circ \gamma}}
W_{kn}(p,q)\\
&=&\sum_{i_1\ldots i_k=1}^n\sum_{\substack{p,q\in D_k\\
p\leq \ker \ii, \gamma(q)\leq \ker \ii}}W_{kn}(p,q)\\
&=&\sum_{p,q\in D_k}n^{\vert p\vee\gamma(q)\vert}W_{kn}(p,q)\\
&=&\sum_{p,q\in D_k}n^{\vert p\vee\gamma(q)\vert} n^{\vert p\vee q\vert-\vert p\vert
-\vert q\vert}(1+O(1/n))
\end{eqnarray*}
The leading order of $n^{\vert p\vee\gamma(q)\vert+\vert p\vee q\vert-\vert p\vert -\vert
q\vert}$ is $n^0$, which is achieved if and only if $q\geq p$ and $p\geq \gamma(q)$, or
equivalently $p=q=\gamma(q)$. This gives the formula \eqref{eq:moments}.

In the classical case, instead of using the approximation for $W_{nk}(p,q)$, we can write
$n^{\vert p\vee\gamma(q)\vert}$ as $G_{nk}(\gamma(q),p)$. (Note that this only makes
sense if we know that $\gamma(q)$ is also an element in $D_k$; and this is only the case
for the classical partition lattices.) Then one can continue as follows:
$$I=
\sum_{p,q\in D_k} G_{nk}(\gamma(q),p) W_{kn}(p,q)= \sum_{q\in D_k}
\delta(\gamma(q),q)=\#\{q\in D_k|q=\gamma(p)\}.
$$
\end{proof}

We discuss now the computation of the asymptotic joint $*$-distribution of the variables
${\rm Tr}(u^k)$. Observe that this is of relevance only in the non-classical context, where
the variables ${\rm Tr}(u^k)$ are in general (for $k\geq 3$) not self-adjoint.

If $c$ is a cycle we use the notation $c^1=c$, and $c^*$= cycle opposite to $c$.

\begin{definition}
Associated to any $k_1,\ldots,k_s\in\mathbb N$ and any $e_1,\ldots,e_s\in\{1,*\}$ is the
\emph{trace permutation} $\gamma\in S_k$, with $k=\Sigma k_i$, having as cycles
$(1,\ldots,k_1)^{e_1}$, $(k_1+1,\ldots,k_1+k_2)^{e_2}$, \ldots,
$(k-k_s+1,\ldots,k)^{e_s}$.
\end{definition}

Observe that with $e_1,\ldots,e_s=1$ we recover the permutation in Definition
\ref{def:trace-permutation}. With this notation, we have the following slight
generalization of Theorem \ref{thm:moments}.

\begin{theorem}
Let $G$ be an easy quantum group. Consider $s\in\NN$, $k_1,\dots,k_s\in\NN$,
$e_1,\ldots,e_s\in\{1,*\}$, $k:=\sum_{i=1}^s k_i$, and denote by $\gamma\in S_k$ the
trace permutation associated to $k_1,\ldots,k_s$ and $e_1,\ldots,e_s$. Then we have, for
any $n$ such that $G_{nk}$ is invertible,
$$\int_G {\rm Tr}(u^{k_1})^{e_1}\ldots {\rm Tr}(u^{k_s})^{e_s}\,du=\#\{p\in
D_k|p=\gamma(p)\}+O(1/n).$$

If $G$ is a classical easy group, then this formula is valid without any lower order
corrections.
\end{theorem}

\begin{proof}
This is similar to the proof of Theorem \ref{thm:moments}.
\end{proof}

\section{Cumulants of powers}

The formula for the moments of the variables ${\rm Tr}(u^k)$ contains in principle all
information about their distribution. However, in order to specify this more explicitly,
in particular, to recognize independence/freeness between those (or suitable
modifications), it is more advantageous to look on the cumulants of these variables. For
this we restrict, in this section, to the classical and free case. We will calculate the
classical cumulants (denoted by $c_r$) for the classical easy groups and the free
cumulants (denoted by $\kappa_r$) for the free easy groups. Actually we will restrict to
the cases
\begin{enumerate}
\item Classical groups: $O_n,S_n,H_n,B_n$.

\item Free quantum groups: $O_n^+,S_n^+,H_n^+,B_n^+$.
\end{enumerate}

The reason for this is that we need some kind of multiplicativity for the underlying
partition lattice in our calculations, as specified in the next proposition.

\begin{proposition}\label{prop:true}
Assume that $G$ is one of the easy quantum groups $O_n,S_n,H_n,B_n$ or
$O_n^+,S_n^+,H_n^+,B_n^+$ and denote by $D_k$ the corresponding category of partitions.
Then we have the following property: let $p\in D_k$ be a partition, and let $q\in P_l$
with $l\leq k$ be a partition arising from $p$ by deleting some blocks. Then $b\in D_l$.
\end{proposition}

\begin{proof}
This follows from the explicit description of the full categories of partitions for the
various easy quantum groups, given in section 1.
\end{proof}

\begin{theorem}\label{thm:cumulants}
1) Let $G$ be one of the easy classical groups $O_n,S_n,H_n,B_n$ with $D_k$ as
corresponding category of partitions. Consider $r\in \mathbb N$, $k_1,\dots,k_r\in
\mathbb N$, $k:=\sum_{i=1}^r k_i$ and $e_1,\dots,e_r\in\{1,*\}$, and denote by $\gamma\in
S_k$ the trace permutation associated to $k_1,\ldots,k_r$ and $e_1,\dots,e_r$. Then we
have, for any $n$ such that $G_{nk}$ is invertible, the classical cumulants
$$c_r({\rm Tr}(u^{k_1})^{e_1},\ldots,{\rm Tr}(u^{k_r})^{e_r})
=\#\{p\in D_k|p\vee\gamma=1_k,\,p=\gamma(p)\}.$$

2) Let $G$ be one of the easy free groups $O_n^+,S_n^+,H_n^+,B_n^+$ with $D_k$ as
corresponding category of non-crossing partitions. Consider $r\in \mathbb N$,
$k_1,\dots,k_r\in \mathbb N$, $k:=\sum_{i=1}^r k_i$ and $e_1,\dots,e_r\in\{1,*\}$, and
denote by $\gamma\in S_k$ the trace permutation associated to $k_1,\ldots,k_r$ and
$e_1,\dots,e_r$. Then we have, for any $n$ such that $G_{nk}$ is invertible, the free
cumulants
$$\kappa_r({\rm Tr}(u^{k_1})^{e_1},\ldots,{\rm Tr}(u^{k_r})^{e_r})
=\#\{p\in D_k|p\vee\gamma=1_k,\,p=\gamma(p)\}+O(1/n).$$
\end{theorem}

\begin{proof}
1) Let us denote by $c_r$ the considered cumulant. We write
$$D_\sigma:=\{p\in P_k\mid p\vert_v \in D_{\vert v\vert}\, \forall v\in\sigma\}$$
for those partitions $p$ in $P_k$ such that the restriction of $p$ to a block of $\sigma$
is an element in the corresponding set $D_{\vert v\vert}$. Clearly, one has that a $p\in
D_\sigma$ is in $D_k$ and must satisfy $p\leq \sigma^\gamma$. Our exclusion of the primed
classical groups guarantees, by Proposition \ref{prop:true}, that this is actually a
characterization, i.e., we have
\begin{equation}\label{eq:D-sigma}
D_\sigma=\{p\in D_k\mid p\leq \sigma^\gamma\}.
\end{equation}

Then, by the definition of the classical cumulants via M\"obius inversion of the moments,
we get from \eqref{eq:moments}:
\begin{align*}
c_r&=\sum_{\sigma\in P(r)}\mu(\sigma,1_r) \cdot \#\{p\in D_\sigma:
p=\gamma(p)\}\\
&=\sum_{\sigma\in P(r)}\mu(\sigma,1_r) \cdot \#\{p\in D_k: p\leq \sigma^\gamma,
p=\gamma(p)\}\\
&=\sum_{\sigma\in P(r)}\mu(\sigma,1_r) \sum_{\substack{p\in D_k\\p\leq \sigma^\gamma,
p=\gamma(p)}}1
\end{align*}
In order to exchange the two summations, we first have to replace the summation over
$\sigma\in P(r)$ by a summation over $\tau:=\sigma^\gamma\in P(k)$. Note that the
condition on the latter is exactly  $\tau\geq \gamma$ and that we have
$\mu(\sigma,1_r)=\mu(\sigma^\gamma,1_k)$. Thus:
$$
c_r =\sum_{\substack{\tau\in P(k)\\ \tau\geq\gamma}}\mu(\tau,1_k) \sum_{\substack{p\in
D_k\\p\leq \tau, p=\gamma(p)}}1=\sum_{\substack{p\in D_k\\
p=\gamma(p)}}\sum_{\substack{\tau\in P(k)\\ p\vee \gamma\leq \tau}}\mu(\tau,1_k)
$$
The definition of the M\"obius function (see (10.11) in \cite{nsp}) gives for the second
summation
$$\sum_{\substack{\tau\in P(k)\\ p\vee \gamma\leq
\tau}}\mu(\tau,1_k)=\begin{cases} 1,& p\vee \gamma=1_k\\0,&\text{otherwise}
\end{cases}$$
and the assertion follows.

2) In the free case, the proof runs in the same way, by using free cumulants and the
corresponding M\"obius function on non-crossing partitions. Note that we have the
analogue of \eqref{eq:D-sigma} in this case only for non-crossing $\sigma$.
\end{proof}

\section{The orthogonal case}

In this section we discuss what Theorem \ref{thm:cumulants} implies for the asymptotic
distribution of traces in the case of the orthogonal quantum groups. For the classical
orthogonal group we will in this way recover the theorem of Diaconis and Shahshahani
\cite{dsh}.

\begin{theorem}\label{thm:orthogonal}
The variables $u_k=\lim_{n\to\infty}{\rm Tr}(u^k)$ are as follows:
\begin{enumerate}
\item For $O_n$, the $u_k$ are real Gaussian variables, with variance $k$
and mean $0$ or $1$, depending on whether $k$ is odd or even. The $u_k$'s are independent.

\item For $O_n^+$, at $k=1,2$ we get semicircular variables of  variance $1$
and mean $0$ for $u_1$ and mean $1$ for $u_2$, and at $k\geq 3$ we get circular
variables of mean $0$ and covariance $1$. The $u_k$'s are $*$-free.
\end{enumerate}
\end{theorem}

\begin{proof}
(1) In this case $D_k$ consists of all pairings of $k$ elements. We have to count all
pairings $p$ with the properties that $p\vee\gamma=1_k$ and $p=\gamma(p)$.

Note that if $p$ connects two different cycles of $\gamma$, say $c_i$ and $c_j$, then the
property $p=\gamma(p)$ implies that each element from $c_i$ must be paired with an
element from $c_j$; thus those cycles cannot be connected to other cycles and they must
contain the same number of elements. This means that for $s\geq 3$ there are no $p$ with
the required properties. Thus all cumulants of order 3 and higher vanish asymptotically
and all traces are asymptotically Gaussian.

Since in the case $s=2$ we only have permissible pairings if the two cycles have the same
number of elements, i.e., both powers of $u$ are the same, we also see that the
covariance between traces of different powers vanishes and thus different powers are
asymptotically independent. The variance of $u_k$ is given by the number of matchings
between $\{1,\dots,k\}$ and $\{k+1,\dots,2k\}$ which are invariant under rotations. Since
such a matching is determined by the partner of the first element 1, for which we have
$k$ possibilities, the variance of $u_k$ is $k$. For the mean, if $k$ is odd there is
clearly no pairing at all, and if $k=2p$ is even then the only pairing of
$\{1,\dots,2p\}$ which is invariant under rotations is $(1,p+1),(2,p+2),\dots,(p,2p)$.
Thus the mean of $u_k$ is zero if $k$ is odd and 1 if $k$ is even.

(2) In the quantum case $D_k$ consists of non-crossing pairings. We can essentially repeat the arguments from above but have to take care that only non-crossing pairings are counted. We also have to realize that for $k\geq 3$, the $u_k$ are not selfadjoint any longer, thus we have to consider also $u_k^*$ in these cases. This means that in our arguments we have to allow cycles which are rotated ``backwards'' under $\gamma$.

By the same reasoning as before we see that free cumulants of order three and higher vanish. Thus we get a (semi)circular family. The pairing which gave mean 1 in the classical case is only in the case $k=2$ a non-crossing one, thus the mean of $u_2$ is 1, all other means are zero. For the variances, one has again that different powers allow no pairings at all and are asymptotically $*$-free. For the matchings between $\{1,\dots,k\}$ and $\{k+1,\dots,2k\}$ one has to observe that there is only one non-crossing possibility, namely $(1,2k),(2,2k-1),\dots,(k,k+1)$ and this satisfies $p=\gamma(p)$ only if $\gamma$ rotates both cycles in different directions.

For $k=1$ and $k=2$ there is no difference between both directions, but for $k\geq 3$
this implies that we get only a non-vanishing covariance between $u_k$ and $u_k^*$ (with
value 1). This shows that $u_1$ and $u_2$ are semicircular, whereas the higher $u_k$ are
circular.
\end{proof}

\section{The bistochastic case}

In the bistochastic case we have the following version of Theorem \ref{thm:orthogonal}.

\begin{theorem}
The variables $u_k=\lim_{n\to\infty}{\rm Tr}(u^k)$ are as follows:
\begin{enumerate}
\item For $B_n$, the $u_k$ are real Gaussian variables, with variance $k$ and
mean $1$ or $2$, depending on whether $k$ is odd or even. The $u_k$'s are independent.

\item For $B_n^+$, at $k=1,2$ we get semicircular variables of  variance $1$
and mean $1$ for $u_1$ and mean $2$ for $u_2$, and at $k\geq 3$ we get circular variables
of mean $1$ and covariance $1$. The $u_k$'s are $*$-free.
\end{enumerate}
\end{theorem}

\begin{proof}
When replacing $O_n$ and $O_n^+$ by $B_n$ and $B_n^+$, we also have to allow singletons
in $p$. Note however that the condition $p=\gamma(p)$ implies that if $p$ has a
singleton, then the corresponding cycle of $\gamma$ must consist only of singletons of
$p$ , which means in particular that this cycle cannot be connected via $p$ to other
cycles. Thus singletons are not allowed for permissible $p$, unless we only have one
cycle of $\gamma$, i.e., we are looking on the mean. In this case there is one additional
$p$, consisting just of singletons, which makes a contribution. So the results for $B_n$
and $B_n^+$ are the same as those for $O_n$ and $O_n^+$, respectively, with the only
exception that all means are shifted by 1.
\end{proof}

\section{The symmetric case}

Let us now consider the case of the symmetric groups. In this case we have to consider
all partitions instead of just pairings and the arguments are getting a bit more
involved. Nevertheless one can treat these cases still in a quite straightforward way.
For the classical permutation groups, one recovers in this way the corresponding result
of Diaconis and Shahshahani \cite{dsh}.

\begin{proposition}\label{prop:S-n}
The cumulants of $u_k=\lim_{n\to\infty}{\rm Tr}(u^k)$ are as follows:
\begin{enumerate}
\item For $S_n$, the classical cumulants are given by:
$$c_r(u_{k_1},\dots ,u_{k_r})=\sum_{q\mid k_i \forall i=1,\dots,r} q^{r-1}$$

\item For $S_n^+$, the free cumulants are given by:
$$c_r(u_{k_1}^{e_1},\dots ,u_{k_r}^{e_r})=\begin{cases}
2,& r=1,\, k_1\geq 2\\
2,& r=2,\, k_1=k_2, \, e_1=e_2^*\\
2,& r=2,\, k_1=k_2=2\\
1,& \text{otherwise.}
\end{cases}$$
\end{enumerate}
\end{proposition}

\begin{proof}
(1) Now $D_k$ consists of all partitions. We have to count partitions $p$ which have the
properties that $p\vee\gamma=1_k$ and $p=\gamma(p)$.

Consider a partition $p$ which connects different cycles of $\gamma$. Consider the
restriction of $p$ to one cycle. Let $k$ be the number of elements in this cycle and $t$
be the number of the points in the restriction. Then the orbit of those $t$ points under
$\gamma$ must give a partition of that cycle; this means that $t$ is a divisor of $k$ and
that the $t$ points are equally spaced. The same must be true for all cycles of $\gamma$
which are connected via $p$, and the ratio between $t$ and $k$ is the same for all those
cycles. This means that if one block of $p$ connects some cycles then the orbit under
$\gamma$ of this block connects exactly those cycles and exhausts all points of those
cycles. So if we want to connect all cycles of $\gamma$ then this can only happen in the
way that we have one (and thus all) block of $p$ intersecting each of the cycles of
$\gamma$. To be more precise, let us consider $c_r(u_{k_1},\dots ,u_{k_r})$. We have then
to look for a common divisor $q$ of all $k_1,\dots,k_r$; a contributing $p$ is then one
the blocks of which are of the following form: $k_1/q$ points in the first cycle (equally
spaced), ... $k_r/q$ points in the last cycle (equally spaced). We can specify this by
saying to which points in the other cycles the first point in the first cycle is
connected. There are $q^{r-1}$ possibilities for such choices. Thus:
$$c_r(u_{k_1},\dots ,u_{k_r})=\sum_{q\mid k_i \forall i=1,\dots,r} q^{r-1}$$

(2) In the quantum permutation case we have to consider non-crossing partitions instead
of all partitions. Most of the contributing partitions from the classical case are
crossing, so do not count for the quantum case. Actually, whenever a restriction of a
block to one cycle has two or more elements then the corresponding partition is crossing,
unless the restriction exhausts the whole group. This is the case $q=1$ from the
considerations above (corresponding to the partition which has only one block), giving a
contribution 1 to each cumulant $c_r(u_{k_1},\dots,u_{k_r})$. For cumulants of order 3 or
higher there are no other contributions. For cumulants of second order one might also
have contributions coming from pairings (where each restriction of a block to a cycle has
one element). This is the same problem as in the $O_n^+$ case; i.e., we only get an
additional contribution for the second order cumulants $c_2(u_k,u_k^*)$. For first order
cumulants, singletons can also appear and make an additional contribution. Taking this
all together gives the formula in the statement.
\end{proof}

In contrast to the two previous cases, the different traces are now not independent/free
any more. Actually, one knows in the classical case that some more fundamental random
variables, counting the number of different cycles, are independent. We can recover this
result, and its free analogue, from Proposition \ref{prop:S-n} in a straightforward way.

\begin{theorem}\label{thm:S-n}
The variables $u_k=\lim_{n\to\infty}{\rm Tr}(u^k)$ are as follows:
\begin{enumerate}
\item For $S_n$ we have a decomposition of type
$$u_k=\sum_{l\mid k}lC_l$$
with the variables $C_k$ being Poisson of parameter $1/k$, and independent.

\item For $S_n^+$ we have a decomposition of the type
$$u_1=C_1,\qquad u_k=C_1+C_k \quad (k\geq 2)$$
where the variables $C_l$ are $*$-free; $C_1$ is free Poisson, whereas $C_2$ is semicircular
and $C_k$, for $k\geq 3$, are circular.
\end{enumerate}
\end{theorem}

Let us first note that the first statement is the result of Diaconis and Shahshahani in
\cite{dsh}. Indeed, the matrix coefficients for $S_n$ are given by
$u_{ij}=\chi(\sigma|\sigma(j)=i)$, and it follows that the variable $C_l$ defined by the
decomposition of $u_k$ in the statement is nothing but the number of $l$-cycles. For a
direct proof for the fact that these variables $C_k$ are indeed independent and Poisson
of parameter $1/k$, see \cite{dsh}. In what follows we present a global proof for (1) and
(2), by using Proposition \ref{prop:S-n}.

\begin{proof}

(1) Let $C_k$ be the number of cycles of length $k$. Instead of writing this in terms of
traces of powers of $u$, it clearer to do it the other way round. We have
$u_k=\sum_{l\mid k} l C_l$. We are claiming now that the $C_k$ are independent and each
is a Poisson variable of parameter $1/k$, i.e., that $c_r(C_{l_1},\dots,C_{l_r})$ is zero
unless all the $l_i$'s are the same, say $=l$, in which case it is $1/l$ (independent of
$r$). This is compatible with the cumulants for the $u_k$, according to:
$$c_r(u_{k_1},\dots ,u_{k_r})=\sum_{l_1\mid k_1}\cdots\sum_{l_r\mid k_r}
l_1\cdots l_r c_r(C_{l_1},\dots,C_{l_r})=\sum_{l\mid k_i\forall i} l^r \frac 1l$$

Since the $C_k$'s are uniquely determined by the $u_k$'s, via some kind of M\"obius
inversion, this shows that also the other way round the formula for the cumulants of the
$u_k$'s implies the above stated formula for the cumulants of the $C_k$'s; i.e., we get
the result that the $C_k$ are independent and $C_k$ is Poisson with parameter $1/k$.

(2) This follows easily from Proposition \ref{prop:S-n}.
\end{proof}

\begin{remark}\label{rem:cycles}
1) In the classical case the random variable $C_l$ can be defined by
\begin{equation}\label{eq:cycle}
C_l=\frac 1l\sum_{\substack{i_1,\dots,i_l\\ \text{distinct}}}u_{i_1i_2}u_{i_2i_3}\cdots
u_{i_li_1}.
\end{equation}
Note that we divide by $l$ because each term appears actually $l$-times, in cyclically
permuted versions (which are all the same because our variables commute).

Note that, by using commutativity and the monomial condition, in general the expression
$u_{i_1i_2}u_{i_2i_3}\cdots u_{i_ki_1}$ has to be zero unless the indices
$(i_1,\dots,i_k)$ are of the form $(i_1,\dots,i_l,i_1,\dots,i_l,\dots)$ where $l$ divides
$k$ and $i_1,\dots,i_l$ are distinct. This yields then the relation
$${\rm Tr}(u^k)=\sum_{i_1,\dots,i_l=1}^n u_{i_1i_2}u_{i_2i_3}\cdots u_{i_li_1}=
\sum_{l\vert k} \sum_{\substack{i_1,\dots,i_l\\ \text{distinct}}}
(u_{i_1i_2}u_{i_2i_3}\cdots u_{i_li_1})^{k/l}= \sum_{l\vert k} l C_l,$$ which we used
before to define the $C_l$. [Note that each $u_{i_1i_2}u_{i_2i_3}\cdots u_{i_li_1}$ is an
idempotent, thus the power $k/l$ does not matter.] This explicit form \eqref{eq:cycle} of
the $C_l$ in terms of the $u_{ij}$ can be used to give a direct proof, by using the
Weingarten formula, of the fact that the $C_l$ are independent and Poisson. We will not
present this calculation here, but will come back to this approach in the case of the
hyperoctahedral group in the next section.

2) In the free case we define the ``cycle'' $C_l$ by requiring neighboring indices to be
different,
\begin{equation}\label{eq:cycle-free}
C_l=\sum_{i_1\not=i_2\not=\dots\not=i_l\not=i_1}u_{i_1i_2}u_{i_2i_3}\cdots u_{i_li_1}.
\end{equation}
Note that if two adjacent indices are the same in $u_{i_1i_2}u_{i_2i_3}\cdots u_{i_li_1}$
then, because of the relation $u_{ij}u_{ik}=0$ for $j\not=k$, all must be the same or the
term vanishes. For the case where all indices are the same we have
$$\sum_{i}u_{ii}u_{ii}\cdots u_{ii}=\sum_i u_{ii}=C_1.$$
This gives then the relation
$${\rm Tr}(u^k)=C_k+C_1.$$
Again, the $C_l$ are uniquely determined by the ${\rm Tr}(u^k)$ and thus our calculations
also show that the $C_l$ defined by \eqref{eq:cycle-free} are $*$-free and have the
distributions as stated.
\end{remark}

\section{The hyperoctahedral case}

The methods in the previous section apply, modulo a grain of salt, as well to the
hyperoctahedral case.

\begin{proposition}
The cumulants of $u_k=\lim_{n\to\infty}{\rm Tr}(u^k)$ are as follows:
\begin{enumerate}
\item For $H_n$, the classical cumulants are given by:
$$c_r(u_{k_1},\dots ,u_{k_r})=\sum_{\substack{q\mid k_i
\forall i=1,\dots,r\\ 2|(\Sigma k_i/q)}} q^{r-1}$$

\item For $H_n^+$, the free cumulants are given by:
$$c_r(u_{k_1}^{e_1},\dots ,u_{k_r}^{e_r})=
2,\qquad\text{if $r=2$, $k_1=k_2$, $e_1=e_2^*$ or if $r=2$, $k_1=k_2=2$}$$ and otherwise
by
$$c_r(u_{k_1}^{e_1},\dots ,u_{k_r}^{e_r})=
\begin{cases}
1,&\text{$\sum_l k_l$ even}\\
0,&\text{$\sum_l k_l$ odd}
\end{cases}
$$

\end{enumerate}
\end{proposition}

\begin{proof}
This follows similarly as the proof of Proposition \ref{prop:S-n}, by taking into account that
we have to restrict attention to the partitions having even blocks only.
\end{proof}

\begin{theorem}\label{thm:hyperoctahedral}
The variables $u_k=\lim_{n\to\infty}{\rm Tr}(u^k)$ are as follows:
\begin{enumerate}
\item For $H_n$ we have a decomposition of type
$$u_k=\sum_{l\mid k}l[C_l^++(-1)^{k/l} C_l^-]$$
with the variables $C_l^+$ and $ C_l^-$ being Poisson of parameter $1/{2l}$, and all
$C_l^+,C_l^-$ ($l\in\mathbb{N}$) being independent.

\item For $H_n^+$ we have a decomposition of type
$$u_1=C_1^+-C_1^-\qquad u_k=C_1^++(-1)^k C_1^- +C_k \quad (k\geq 2)$$
where $C_1^+$, $C_1^-$ and $C_k$ ($k\geq 2$) are $*$-free and $C_1^+$, $C_1^-$ are free
Poisson elements of parameter $1/2$, $C_2$ is a semicircular element, and $C_k$ ($k\geq
3$) are circular elements.

\end{enumerate}
\end{theorem}

\begin{proof}
(1) This follows in the same way as in the proof of Theorem \ref{thm:S-n}.

(2) This follows by direct computation.
\end{proof}

In the classical case the random variables $C_l^+$ and $C_l^-$ should count the number of
positive cycles of length $l$ and the number of negative cycles of length $l$,
respectively, and should be given by $C_l^+=Z_l^+$ and $C_l^-=Z_l^-$ with
\begin{equation}\label{eq:cycle-positive}
Z_l^+=\frac 1l\sum_{\substack{i_1,\dots,i_l\\\text{distinct}}}
1_{\{1\}}(u_{i_1i_2}u_{i_2i_3}\cdots u_{i_li_1})
\end{equation}
 and
\begin{equation}\label{eq:cycle-negative}
Z_l^-=-\frac 1l\sum_{\substack{i_1,\dots,i_l\\\text{distinct}}}
1_{\{-1\}}(u_{i_1i_2}u_{i_2i_3}\cdots u_{i_li_1}).
\end{equation}

Note that $u_{i_1i_2}u_{i_2i_3}\cdots u_{i_li_1}$ is either -1, 0, 1. $1_{\{1\}}$ denotes
the characteristic function on 1 and $1_{\{-1\}}$ the characteristic function on $-1$. As
in Remark \ref{rem:cycles} it follows that one has the decomposition
\begin{equation}\label{eq:decomposition-hyper}
\Tr(u^k)=\sum_{l\mid k}l[Z_l^++(-1)^{k/l} Z_l^-]
\end{equation}

Again one expects that the random variables defined by \eqref{eq:cycle-positive} and
\eqref{eq:cycle-negative} are all independent and are Poisson, but because now the
$Z_l^+$, $Z_l^-$ are not uniquely determined by the relations
\eqref{eq:decomposition-hyper}, we cannot argue that the $C_l^+$ and $C_l^-$ showing up
in the decomposition in Theorem \ref{thm:hyperoctahedral} are the same as $Z_l^+$ and
$Z_l^-$ defined by \eqref{eq:cycle-positive} and \eqref{eq:cycle-negative}. In order to
see that this is actually the case, we will now, in the following, calculate the
cumulants of the random variables defined by \eqref{eq:cycle-positive} and
\eqref{eq:cycle-negative}.

Note first that we can replace the characteristic functions in the following way:
\begin{equation}
2Z_l^+=\frac 1l\sum_{\substack{i_1,\dots,i_l\\\text{distinct}}}
(u_{i_1i_2}u_{i_2i_3}\cdots u_{i_li_1})^2+u_{i_1i_2}u_{i_2i_3}\cdots u_{i_li_1}
\end{equation}
 and
\begin{equation}
2Z_l^-=\frac
1l\sum_{\substack{i_1,\dots,i_l\\\text{distinct}}}(u_{i_1i_2}u_{i_2i_3}\cdots
u_{i_li_1})^2-u_{i_1i_2}u_{i_2i_3}\cdots u_{i_li_1}.
\end{equation}

Furthermore, we have
$$\sum_{\substack{i_1,\dots,i_l\\\text{distinct}}} u_{i_1i_2}u_{i_2i_3}\cdots u_{i_li_1}=
\sum_{\ker \ii=1_l} u_{i_1i_2}\cdots u_{i_{l}i_1}$$ and
$$\sum_{\substack{i_1,\dots,i_l\\\text{distinct}}} (u_{i_1i_2}u_{i_2i_3}\cdots u_{i_li_1})^2=
\sum_{\ker \ii=\tau_l^{2l}} u_{i_1i_2}\cdots u_{i_{2l}i_1},$$ where $\tau_l^{2l}\in
P(2l)$ is the pairing $\{(1,l+1),(2,l+2),\dots,(l,2l)\}$.

So what we need are cumulants for general variables of the form
$$Z(\sigma):=\sum_{\ker \ii=\sigma} u_{i_1i_2}\cdots u_{i_{l}i_1},$$
for arbitrary $l\in\NN$ and $\sigma\in P(l)$. Note that $Z(\sigma)$ can only be different
from zero if $\sigma$ is of the form $\sigma=\tau_l^k$ for $k,l\in \NN$ with $l\vert k$,
where
$$\tau_l^k=\{(1,l+1,2l+1,\dots,k-l+1),(2,l+2,2l+2,\dots,k-l+2),\dots,(l,2l,\dots,k)\}.$$
For $k=l$, we have $\tau_l^l=1_l$.

\begin{theorem}\label{thm:c-Z-variables}
For all $s\in\NN$, $k_1,\dots,k_s\in\NN$, $\sigma_1\in P(k_1),\dots, \sigma_s\in P(k_s)$
we have
\begin{align*}
c_r[Z(\sigma_1),\dots,Z(\sigma_r)]=\#\{q\in D_k\mid &\, q=\gamma(q), q\vee \gamma=1_k,\\
&\text{$q$ restricted to the $i$-th cycle of $\gamma$ is $\sigma_i$} (i=1,\dots,r)\},
\end{align*}
where $k=\sum_{i=1}^r k_i$ and $\gamma$ is the trace permutation associated to
$k_1,\dots,k_r$
\end{theorem}

Note that also the right hand side of the equation is, by the condition $q=\gamma(q)$,
zero unless all $\sigma$ are of the form $\tau_l^k$

\begin{proof}
Let us first calculate the corresponding moment. For this we note that one has
$$\sum_{\ker \ii\geq\pi} u_{i_1i_2}\cdots u_{i_{l}i_1}=\sum_{\substack{\sigma\in P(l)\\
\sigma\geq\pi}} Z(\sigma),$$ and thus, by M\"obius inversion on $P(l)$
$$Z(\sigma)=\sum_{\substack{\pi\in P(l)\\ \pi \geq \sigma}}\mu(\sigma,\pi)\sum_{\ker \ii\geq \pi}
u_{i_1i_2}\cdots u_{i_{l}i_1}.$$ With this we can calculate
\begin{align*}
&\int_{H_n}Z(\sigma_1)\cdots Z(\sigma_r) du=
\sum_{\pi_1,\dots,\pi_r}\mu(\sigma_1,\pi_1)\cdots
\mu(\sigma_r,\pi_r)\sum_{\pi_1\circ\cdots\circ \pi_r\leq\ker \ii} \int_{H_n}
u_{i_1i_{\gamma(1)}}\cdots u_{i_ki_{\gamma(k)}}\\
&=\sum_{\pi_1,\dots,\pi_r}\mu(\sigma_1,\pi_1)\cdots
\mu(\sigma_r,\pi_r)\sum_{\pi_1\circ\cdots\circ \pi_r\leq\ker \ii}\sum_{\substack{q,p\in
D_k\\
p\leq\ker \ii, \gamma(q)\leq \ker \ii}}W_{kn}(p,q)\\
&=\sum_{\pi_1,\dots,\pi_r}\mu(\sigma_1,\pi_1)\cdots \mu(\sigma_r,\pi_r)\sum_{q,p\in D_k}
G_{kn}(\gamma(q)\vee \pi_1\circ\cdots\circ \pi_r,p) W_{kn}(p,q)\\
&=\sum_{\pi_1,\dots,\pi_r}\mu(\sigma_1,\pi_1)\cdots \mu(\sigma_r,\pi_r)\sum_{q\in D_k}
\delta(\gamma(q)\vee \pi_1\circ\cdots\circ \pi_r,q)\\
\end{align*}
In the third line, it looks as if we might have a problem because $\pi_1\circ\cdots\circ
\pi_r$ is in $P_k$, but not necessarily in $D_k$. However, our category $D_k$ has the
nice property that, for $\pi\in P_k$ and $\sigma\in D_k$, $\pi\geq\sigma$ implies that
also $\pi\in D_k$. Thus in particular, $\pi\vee\sigma\in D_k$ for any $\pi\in P_k$ and
$\sigma\in D_k$, and we have in our case that always $\gamma(q)\vee \pi_1\circ\cdots\circ
\pi_r\in D_k$. Now note further that $\gamma(q)\vee \pi_1\circ\cdots\circ \pi_r=q$ is
actually equivalent to $\gamma(q)=q$ and $\pi_1\circ\cdots\circ \pi_r\leq q$. One
direction is clear, for the other one has to observe that $\gamma(q)\leq q$ implies
$\gamma(q)=q$. With this, we get finally
\begin{multline}
\int_{H_n}Z(\sigma_1)\cdots Z(\sigma_r) du\\= \#\{q\in D_k\mid q=\gamma(q), \text{$q$
restricted to the $i$-th cycle of $\gamma$ is $\sigma_i$} (i=1,\dots,r)\}
\end{multline}

From this, we get for the cumulants
\begin{align*}
&c_r[Z(\sigma_1),\dots,Z(\sigma_r)]\\&=\sum_{\pi\in P(r)}\mu(\pi,1_r)\cdot \#\{q\in D_k
\mid
q=\gamma(q),q\vee\gamma\leq\pi^\gamma,\\
&\qquad\qquad\qquad\qquad\qquad\qquad\qquad\text{$q$ restricted to the $i$-th cycle of
$\gamma$ is $\sigma_i$} (i=1,\dots,r)\}.
\end{align*}
The result follows then from M\"obius inversion, as in the proof of Theorem
\ref{thm:cumulants}.
\end{proof}

This shows in particular that $c_r[Z(\sigma_1),\dots,Z(\sigma_r)]$ vanishes unless all
$\sigma_1,\dots,\sigma_r$ have the same number of blocks. This implies that the sets
$\{Z_l^+,Z_l^-\}$ are independent for different $l$. For fixed $l$, we have for all
$e_1,\dots,e_r\in\NN$:
$$
c_r[Z(\tau_{l}^{e_1l}),\dots,Z(\tau_l^{e_rl})]=
\begin{cases}
l^{r-1},&\text{if $\sum_i e_i$ is even}\\
0,&\text{otherwise} \end{cases}.
$$
Thus we get in particular
\begin{multline*}
c_r[Z(\tau_{l_1}^{2l_1})\pm Z(1_{l_1}),\dots,Z(\tau_{l_r}^{2l_r})\pm Z(1_{l_r})]\\=
\begin{cases}
d_{r,l},&\text{if $l_1=\dots=l_r=l$ and all signs are either $+$ or all are $-$}\\
0,&\text{otherwise} \end{cases}
\end{multline*}
where $$d_{r,l}= l^{r-1}\sum^r_{\substack{t=0\\ \text{$t$ even}}} \binom r t=l^{r-1}
2^{r-1}.$$ This shows that also $Z_l^+$ and $Z_l^-$ are independent, and each of them is
Poisson of parameter $1/(2l)$.

\begin{corollary}
In $H_n$ the random variables $Z_l^+$, $Z_l^-$ ($l\in\NN$), defined by
\eqref{eq:cycle-positive} and \eqref{eq:cycle-negative}, are independent Poisson
variables of parameter $1/(2l)$.
\end{corollary}

\begin{remark}
In the free case $H_n^+$, the variables $u_{ii}$ have also spectrum $\{-1,0,1\}$ and we
can consider a positive/negative decomposition for $C_1$, i.e.,
$$C_1^+=\sum_{i}1_{\{1\}}(u_{ii})$$
and
$$C_1^-=-\sum_{i}1_{\{-1\}}(u_{ii});$$
the other $C_l$, $l\geq 2$, are just as in the case of $S_n^+$. Similarly as for $H_n$,
one can show that these variables are the ones showing up in the decomposition for
$H_n^+$ in Theorem \ref{thm:hyperoctahedral}.

\end{remark}

\section{The half-liberated cases}
The half-liberated quantum groups $O_n^*$ and $H_n^{(s)}$ are neither classical nor free
groups, so both classical and free cumulants are inadequate tools for getting information
on the distribution of their traces. In \cite{bcs2}, we introduced half-liberated
cumulants to deal with half-independence, but one has to realize that we do not get an
analogue of Theorem \ref{thm:cumulants} for them, because the underlying ``balanced''
partition lattices do not share the multiplicativity property from Proposition
\ref{prop:true}. In order to investigate the distribution of traces in the half-liberated
cases we will thus have to proceed via another route. The key insight is here that the
half-liberated situations are actually ``orthogonal'' versions of classical unitary
groups and that the main computations can be done over these unitary groups instead.

\subsection{The half-liberated orthogonal group $O_n^*$}
Let $u=(u_{ij})_{i,j=1}^n$ be the fundamental representation of $O_n^*$, and let
$v=(v_{ij})_{i,j=1}^n$ be the fundamental representation of the unitary group $U_n$. Then
we can ``orthogonalize'' $U_n$ by considering
\begin{equation}\label{eq:orthogonalizing}
w_{ij}:=\begin{pmatrix}
0& v_{ij}\\
\overline{v_{ij}}&0\end{pmatrix}.
\end{equation}
Then $w=(w_{ij})_{i,j=1}^n$ is an orthogonal matrix and a simple calculation shows that
the $w_{ij}$ half-commute. It is also easy to see (by invoking the Weingarten formula for
$U_n$, see below, Eq. \eqref{eq:Weingarten-Un}), that under this map the Haar state on
$O_n^*$ goes to $\int_{U_n}\otimes {\rm tr}_2$ . Since the Haar state on $O_n^*$ is
faithful \cite{bve} , the mapping $u_{ij}\mapsto w_{ij}$ is actually an isomorphism.

So we have
$$\Tr(u^{2k+1})=
\begin{pmatrix}
0& \Tr\bigl((v\bar v)^k v\bigr)\\
\Tr\bigl((\bar v v)^k \bar v\bigr)&0
\end{pmatrix}$$ and
$$\Tr(u^{2k})=
\begin{pmatrix}
\Tr\bigl((v\bar v)^k\bigr)&0\\
0&\Tr\bigl((\bar v v)^k\bigr)
\end{pmatrix}$$

So what we need the understand is the distribution of the variables
\begin{equation}\label{eq:Un}
v_{2k+1}:=\lim_{n\to\infty}\Tr\bigl((v\bar v)^k v\bigr),\qquad
v_{2k}:=\lim_{n\to\infty}\Tr\bigl((v\bar v)^k\bigr).\end{equation}

\begin{proposition}
Let $(v_k)_{k\geq 1}$ be as in \eqref{eq:Un}, where $v=(v_{ij})_{i,j=1}^n$ are the
coordinates of the classical unitary group $U_n$. Then we have: the variables
$(v_k)_{k\geq 1}$ are independent; for $k$ even, $v_k$ is a real Gaussian with mean 0 or
1, depending on whether $k/2$ is odd or even, and variance equal to $k/2$; for $k$ odd,
$v_k$ is a complex Gaussian with mean 0 and variance 1.
\end{proposition}

\begin{proof}
For $\epsilon=(e_1,\dots,e_k)$ a string of 1's and $*$'s, let us denote by $P_2(\ee)$ the
pairings in $P_k$ such that each block joins a 1 and a $*$. Then the Weingarten formula
for $U_n$ \cite{...} says that with the notation
$$u_{ij}^e:=\begin{cases}
u_{ij},&\text{if $e=1$}\\
\overline{u_{ij}},&\text{if $e=*$}\end{cases}$$ we have
\begin{equation}\label{eq:Weingarten-Un}
\int_{U_n} u_{i_1j_1}^{e_1}\cdots u_{i_kj_k}^{e_k}du=\sum_{\substack{p,q\in P_2(\ee)\\
p\leq \ker \ii,q\leq \ker \ii}}W(p,q),
\end{equation}
where $\ee=(e_1,\dots,e_k)$.

As in the proof of Theorem \ref{thm:moments} this implies that
\begin{equation}\label{eq:moments-On}
\int v_{k_1}^{e_1}\cdots v_{k_s}^{e_s}=\#\{p\in P_2(\ee)\mid p=\gamma(p)\}+O(1/n).
\end{equation}

The condition $p=\gamma(p)$ implies that the pairing $p$ cannot join two cycles of
different lengths, which shows that such an expectation factorizes according to the cycle
lengths, which implies the independence of the $v_k$. The statements on the distribution
of $v_k$ follow also immediately from \eqref{eq:moments-On}.

\end{proof}

Transferring these results from the $v_k$ to
\begin{equation}\label{eq:orth-traces}
u_{2k+1}=
\begin{pmatrix}
0& v_{2k+1}\\
\overline{v_{2k+1}}&0
\end{pmatrix}, \qquad
u_{2k}=
\begin{pmatrix}
v_{2k}&0\\
0&\overline{v_{2k}}
\end{pmatrix}
\end{equation}
and noting that the distribution of $u_{2k+1}$ is equal to that of $\sqrt{\vert
v_{2k+1}\vert^2}$ (which is a symmetrized Rayleigh distribution) yields then the
following result. (See \cite{bcs2} for the notion of ``half-independence''.)

\begin{theorem}
For $O_n^*$, the variables $u_k=\lim_{n\to\infty}{\rm Tr}(u^k)$ are as follows. The sets
$\{u:k\mid \text{$k$ odd}\}$ and $\{u_k\mid\text{$k$ even}\}$ are independent; for $k$
even, the $u_k$ are independent real Gaussian of mean 0 or 1, depending on whether $k/2$
is even or odd, and variance $k/2$; for $k$ odd, the $v_k$ are half-independent
symmetrized Rayleigh variables with variance 1.
\end{theorem}

\subsection{The hyperoctahedral series $H_n^{(s)}$}
The hyperoctahedral series $H_n^{(s)}$ (for $s=2,3,\dots,\infty$) is determined by the
partition lattice of all $s$-balanced partitions, see Theorem \ref{thm:hyper-series}.
This series includes the classical hyperoctrahedral group for $s=2$, $H_n^{(2)}=H_n$, and
the half-liberated hyperoctahedral group for $s=\infty$, $H_n^{(\infty)}=H_n^*$. As for
$O_n^*$, these groups can be considered as orthogonal versions of classical unitary
groups. Namely, let $H_n^s=\mathbb Z_s\wr S_n$ be the complex reflection group consisting
of monomial matrices having the $s$-roots of unity as nonzero entries. (Note that for
$s=2$, $H_n^{(2)}=H_n^{2}$.) Then the relation between $H_n^{(s)}$ and $H_n^{s}$ is the
same as the one between $O_n^*$ and $U_n$, i.e., we can represent the coordinates
$u_{ij}$ of $H_n^{(s)}$ by $w_{ij}$ according to \eqref{eq:orthogonalizing}, where
$v_{ij}$ are the coordinates of $H_n^{s}$. So again, we can realize the asymptotic traces
$u_k$ of $H_n^{(s)}$ in the form \eqref{eq:orth-traces}, where the $v_k$ are now the
asymptotic traces in $H_n^{s}$ according to \eqref{eq:Un}. So our main task will be the
determination of the distribution of these $v_k$.

Actually, we can treat more generally asymptotic traces with arbitrary pattern of the
conjugates. So let us consider for an arbitrary string $\ee=(e_1,\dots,e_k)$ of $1$ and
$*$ the variable
$$v(\ee):=\lim_{n\to\infty}\Tr(v^{e_1}\cdots v^{e_k}).$$

For $\epsilon=(e_1,\dots,e_k)$ a string of 1's and $*$'s, we denote by $P^s(\ee)$ the
partitions in $P_k$ such that each block joins the same number, modulo $s$, of 1 and $*$.
Then the Weingarten formula for $H^{s}$ says that with the notation
$$v_{ij}^e:=\begin{cases}
v_{ij},&\text{if $e=1$}\\
\overline{v_{ij}},&\text{if $e=*$}\end{cases}$$ we have for the coordinate functions
$v=(v_{ij})_{i,j=1}^n$ of $H_n^s$ that
$$\int_{H_n^{s}} v_{i_1j_1}^{e_1}\cdots v_{i_kj_k}^{e_k}du=
\sum_{\substack{\pi,\sigma\in P^s(\bee)\\ \pi\leq \ker \ii\\ \sigma\leq \ker
\ii}}W_{\ee,n}(\pi,\sigma),$$ where $\ee=(e_1,\dots,e_k)$ and $W_{\ee,n}$ is the inverse
of the Gram matrix $G_{\ee,n}=(n^{\vert p\vee q\vert})_{p,q\in P^s(\ee)}$. The leading
order in $n$ of the Weingarten function $W_{\ee,n}$ is given by
$$W_{\ee,n}(p,q)=n^{\vert p\vee
q\vert -\vert p\vert - \vert q\vert}\bigl(1+O(1/n)\bigr)$$

\begin{theorem}\label{thm:cumulants-Hs}
Fix $s\in\{2,3,\dots,\infty\}$ and consider $H^{s}$. Consider $r\in \mathbb N$,
$k_1,\dots,k_r\in \mathbb N$, and denote by $\gamma\in S_k$ the trace permutation
associated to $k_1,\ldots,k_r$. Then, for any strings $\ee_1,\dots,\ee_r$ of respective
lengths $k_1,\dots,k_r$ we have the classical cumulants
$$c_r(v(\ee_1),\ldots,v(\ee_r))
=\#\{p\in P^s(\ee_1\cdots\ee_r)|p\vee\gamma=1_k,\,p=\gamma(p)\},$$ where the product of
strings is just given by their concatenation.
\end{theorem}

\begin{proof}
As in the proof of Theorem \ref{thm:moments} one gets for the moments
$$\int_{H_n^s} \Tr(v(\ee_1))\cdots\Tr(v(\ee_r))dv
=\#\{p\in P^s(\ee_1\cdots\ee_r)|\,p=\gamma(p)\}+O(1/n).$$ Note that $\gamma$ does not
necessarily map $P^s(\ee_1\cdots\ee_r)$ into itself, and thus we only get the asymptotic
version with lower order corrections.

We can then repeat the proof of Theorem \ref{thm:cumulants}. Let us write $\ee$ for
$\ee_1\cdots\ee_r$; then we only have to note that $P^s(\ee\vert_v)$ (for a block
$v\in\sigma$) records the information about the original positions of the 1 and $*$ in
$\ee$, and thus the multiplicativity issue which prevented us from extending Theorem
\ref{thm:cumulants} to all easy classical groups, is not a problem here. Indeed, we have
the analogue of \eqref{eq:D-sigma},
$$P^s(\ee)_\sigma:=\{p\in P_k\mid p\vert_v \in P^s(\ee\vert_v) \,\forall v\in\sigma\}
=\{p\in P^s(\ee)\mid p\leq \sigma^\gamma\}.
$$

\end{proof}

Again, one can reduce the traces to more basic ``cycle'' variables. As before, we denote,
for $l\vert k$, by $\tau_l^k\in P_k$ the partition
$$\tau_l^k=\{(1,l+1,2l+1,\dots,k-l+1),(2,l+2,2l+2,\dots,k-2+2),\dots,(l,2l,\dots,k)\}$$
Then we have
\begin{equation}\label{eq:vee}
v(\ee)=\sum_{l\mid k} Z(\tau_l^k,\ee),
\end{equation}
where
\begin{equation}
Z(\tau_l^k,\ee):=\sum_{\ker \ii=\tau_l^k} v_{i_1i_2}^{e_1}\cdots v_{i_{k}i_1}^{e_k},
\end{equation}
for arbitrary $l,k\in\NN$ with $l\vert k$, and $\ee=(e_1,\dots,e_k)$.

As in the proof of Theorem \ref{thm:c-Z-variables} we can show
\begin{align}\label{eq:cr-Z}
c_r[Z(\tau_{l_1}^{k_1},\ee_1),&\dots,Z(\tau_{l_r}^{k_r},\ee_r)]\\ \notag &= \#\{p\in
P^s(\ee_1\cdots \ee_r)\mid \, p=\gamma(p), p\vee \gamma=1_k,\\
\notag&\qquad\qquad\text{$p$ restricted to the $i$-th cycle of $\gamma$ is
$\tau_{l_i}^{k_i}$} (i=1,\dots,r)\}.
\end{align}
Clearly, this is only different from zero if $l_1=\cdots=l_r$.

Let us define random variables $\Cp(\tau_l^k,\ee)$ by specifying their distribution as
\begin{multline}\label{eq:def-Cp}
c_r[\Cp(\tau_{l_1}^{k_1},\ee_1),\dots,\Cp(\tau_{l_r}^{k_r},\ee_r)]\\
=\begin{cases} 1/l,&\text{if $l_1=\dots=l_r=l$ and
$p(\tau_{l}^{k_1},\dots,\tau_{l}^{k_r})\in P^s(\ee_1\cdots \ee_r)$}\\
0,&\text{otherwise}
\end{cases}
\end{multline}
where $p(\tau_{l}^{k_1},\dots,\tau_{l}^{k_r})$ is that partition in $P_k$ whose $i$-th
block consists of the union of the $i$-th blocks of all the $\tau$'s, i.e., it is equal
to $$ \tau_{l}^{k_1}\circ\dots\circ\tau_{l}^{k_r} \vee
\{(1,k_1+1,k_1+k_2+1,\dots,k-k_r+1),\dots,(l,k_1+l,k_1+k_2+l,\dots,k-k_r+l)\}.
$$

Then we can express our variables $Z(\tau_l^k,\ee)$ in terms of the $\Cp(\tau_l^k,\ee)$
by
$$Z(\tau_l^k,\ee)=\sum_{t=1}^{l} \Cp(\tau_l^k,\ee^{(t)}),$$
where $\ee^{(t)}$ is the $t$-fold cyclic shift of the string $\ee$, i.e.,
$$\ee^{(t)}=(e_{t+1},e_{t+2},\dots,e_{t})$$

The definition \eqref{eq:def-Cp}, on the other hand, shows that the variables
$\Cp(\tau_l^k,\ee)$ are compound Poisson elements, which are independent for different
$l$. Namely, we can associate to $\Cp(\tau_l^k,\ee)$ a random variable
$$a(\tau_l^k,\ee)=\left(\prod_{\substack{\text{$i_1$ in first}\\\text{block of $\tau_l^k$}}}
\omega^{e_{i_1}}\right)\otimes \left(\prod_{\substack{\text{$i_2$ in second}\\
\text{block of $\tau_l^k$}}}\omega^{e_{i_2}}\right)\otimes\cdots\otimes
\left(\prod_{\substack{\text{$i_{l}$ in $l$-th}\\ \text{block of
$\tau_l^k$}}}\omega^{e_{i_{l}}}\right)\in C(\TT)^{\otimes l}.$$ Then we have
\begin{equation}\label{eq:c-Cp}
c_r[\Cp(\tau_{l_1}^{k_1},\ee_1),\dots,\Cp(\tau_{l_r}^{k_r},\ee_r)]= \frac 1{l_1}
\psi\bigl(a(\tau_{l_1}^{k_1},\ee_1)\cdots a(\tau_{l_r}^{k_r},\ee_r)\bigr)
\end{equation}
where $\psi$ is $\bigoplus_l \ff^{\otimes l}$ on $\bigoplus_l C(\TT)^{\otimes l}$, with
$\ff$ denoting integration with respect to the Haar measure on $\ZZ_s$ (where the latter
is being embedded into the unit circle $\TT$).

The equation \eqref{eq:c-Cp} shows that the cumulants of the variables $\Cp$ are given,
up to some factor, as the corresponding moments of some variables $a$; this is the
characterizing property of compound Poisson variables.

If we put now $$\ee_k:=(\underbrace{1,*,1,*,\dots}_{k}),$$ then we have for our
asymptotic traces the decomposition
\begin{equation}\label{eq:v-k-decomp}
v_k=\sum_{l\mid k} Z(\tau_l^k,\ee_k)=\sum_{l\vert k} \sum_{t=1}^{l}
\Cp(\tau_l^k,\ee_k^{(t)}).
\end{equation}
Thus we have written the $v_k$ as a sum of compound Poisson variables. These are
independent for different $l$; however, for fixed $l$, the relation between the $\Cp$ for
various $k$ is more complicated, according to the $\ee$-strings. For $k$ even this
reduces again essentially to a sum of independent Poisson variables, whereas for $k$ odd
the situation is getting more involved. As we do not see a nice more explicit
description, we refrain from working out the details in this case.

\section{Concluding remarks}

We have seen in this paper that the original philosophy suggested in \cite{bsp}, namely
the fact that ``any result which holds for $S_n,O_n$ should have an extension to easy
quantum groups'', has indeed a first illustration in the context of the fundamental
stochastic eigenvalue computations of Diaconis and Shahshahani in \cite{dsh}.

A first natural question is about the possible ``eigenvalue'' interpretations of our
results. The point is that in the classical case the law of ${\rm Tr}(u^k)$ is of course
a function of the eigenvalues of the random matrix $u\in G$, but in the quantum case such
a simple interpretation is lacking. One technical problem is for instance the fact that
the variables ${\rm Tr}(u^k)$ are not self-adjoint in the quantum case. So, as a first
conclusion of our study, we would like to point out the fact that the following question
is still open: \emph{What are the eigenvalues of a random quantum group matrix?}

\begin{remark}
The half-liberated cases are, as we have seen in the last section, quite close to the
classical world, and the representation in terms of $2\times 2$-matrices over classical
unitary groups suggests an answer to the above question in this case: If $\lambda_i$ are
the eigenvalues of $U_n$, then
$$\begin{pmatrix}
0 &\lambda_i\\
\bar \lambda_i &0
\end{pmatrix}
$$ should be the corresponding eigenvalues of $O_n^*$. The same for $H^{s}$ and
$H^{(s)}$.

For the free quantum groups, however, the situation is less clear and we have no
suggestion for a possible candidate for eigenvalues.
\end{remark}

A second natural question is about what happens in the unitary case. It is known indeed
since \cite{ban} that the fundamental character $\chi={\rm Tr}(u)$ is asymptotically
circular in the sense of Voiculescu \cite{vdn} for the quantum group $U_n^+$, so the
results in \cite{dsh} about $U_n$ should probably have some kind of ``free version''.
However, the general study and classification of the easy quantum groups in the unitary
case seems to be a quite difficult combinatorial problem, and we do not have so far
concrete results in this direction. We would like to refer here to the concluding section
in our previous paper \cite{bsp}, which contains a brief description of the whole
problematics in the unitary case.

A third question concerns the relationship of the present results with the Bercovici-Pata
bijection \cite{bpa}. This bijection makes a correspondence between classical measures
and their ``free versions'', and a key problem in free probability is to find concrete
models for it. A general random matrix model, providing a full answer to the question,
comes from the work of Benaych-Georges \cite{ben} and Cabanal-Duvillard \cite{cab}. As
for the representation theory implementations, meant to be ``finer'', these concern so
far the laws of truncated characters for $O_n,S_n,H_n,B_n$, as explained in \cite{bsp}.
So, the question that we would like to raise here is as follows: is it possible to unify
the truncated character computations in \cite{bsp} with the Diaconis-Shahshahani type
computations from the present paper? This is definitely possible in the classical case,
where several ``truncation'' procedures are available in the general context of
stochastic eigenvalue analysis. See e.g. \cite{nov}
 .

A fourth fundamental problem is about what happens when $n$ is fixed. In the classical
case the subject is of course quite technical, but the results in this sense abound. In
the quantum case the situation is  definitely more complicated, because the main tool
that we have so far, namely the Weingarten formula, cannot properly handle the problem.
In the $O_n^+$ case the law of $u_{11}$, which can be regarded as a truncated character,
was computed only recently, in \cite{bcz}. We do not know if the techniques developed
there, which are new, can be applied to the variables investigated in the present paper.

Finally, we have the general question of trying to apply our ``$S_n,O_n$ philosophy'' to
some new, totally different situations. The first thought here goes to the various de
Finetti type theorems, available for $S_n,O_n$ and other classical groups from
\cite{kal}, and for $S_n^+,O_n^+$ from \cite{ksp} and \cite{cu1}, \cite{cu2}. A global
approach to the problem, by using easy quantum groups, is developed in our paper
\cite{bcs2}.

\end{document}